\documentclass[a4paper,10pt]{article}
\usepackage{amsmath}
\usepackage{amsfonts}
\usepackage{graphicx}
\usepackage{latexsym, amscd, euscript}
\usepackage{color}
\usepackage{graphics}
\usepackage[all,2cell]{xy}
\usepackage{amssymb}
\usepackage[margin=3.5cm]{geometry}
\usepackage{mathrsfs}
\usepackage{float}
\usepackage{makeidx}
\makeindex

\usepackage[compact]{titlesec}
\usepackage{fancyhdr}

\xyoption{import}

\makeindex

\titleformat{\chapter}[display]
{\bfseries\Huge} {\vspace{-2cm}\filleft\Huge \chaptertitlename
\hspace{.5cm}\thechapter} {2ex} {\titlerule[1pt] \vspace{1pt}
\titlerule
\vspace{1ex}%
\filright}
[\vspace{1ex}%
\titlerule]

\titleformat{\section}[hang]
{\Large\sffamily}
{\thesection.--}
{2pt}
{\vspace{2pt}\Large}

\fancypagestyle{plain}{%
\fancyhf{} 
\fancyfoot[C]{\thepage} 

}
\usepackage{amsthm}
\usepackage[dvips,final]{epsfig}
\usepackage[english]{babel}

\title{Puiseux Parametric Equations via the Amoeba of the Discriminant \footnotetext{This work has been supported by PAPIIT IN104713 y IN108216, ECOS NORD M14M03, LAISLA and CONACYT.}}
\author{Fuensanta Aroca and V\'ictor Manuel Saavedra}

\date{}
\medskip
\medskip
\begin{document}
\maketitle

\begin{flushright}
\emph{Dedicated to Jose Seade on his $60^{th}$ birthday}
\end{flushright}

\begin{abstract}Given an algebraic variety we get Puiseux type parametrizations on suitable Reinhardt domains. These domains are defined using the amoeba of hypersurfaces containing the discriminant locus of a finite projection of the variety.
\end{abstract}
\newtheorem{ej}{Example}
\newtheorem{thm}{Theorem}
 \newtheorem{cor}{Corollary}
 \newtheorem{lem}{Lemma}
 \newtheorem{prop}{Proposition}
 \newtheorem{defn}{Definition}
\newtheorem{rem}{Remark}

\section{Introduction} The theory of complex algebraic or analytic singularities is the study of systems of a finite number of equations in the neighborhood of a point where the rank of the Jacobian matrix is not maximal. These points are called singular points.

Isaac Newton in a letter to Henry Oldenburg \cite{chi3}, described an algorithm to compute term by term local parameterizations at singular points of plane curves. The existence of such parameterizations (i.e. the fact that the algorithm really works) was proved by Puiseux \cite{chi4} two centuries later. This is known as the Newton-Puiseux theorem and asserts that we can find local parametric equations of the form $z_{1}=t^{k}$, $z_{2}=\varphi(t)$ where $\varphi$ is a convergent power series.

Singularities of dimension greater than one are not necessarily parameterizable. An important class of parameterizable singularities are called quasi-ordinary. 

S.S. Abhyankar proved in \cite{Ab} that quasi-ordinary hypersurface singularities are parameterizable by Puiseux series.\\
 For algebraic hypersurfaces J. McDonald showed in \cite{Mc} the existence of Puiseux series solutions with support in strongly convex cones. P.D. Gonz\'alez P\'erez \cite{Go} describes these cones in terms of the Newton polytope of the discriminant. In \cite{Ar} F. Aroca extends this result to arbitrary codimension.

In this paper we prove that, for every connected component of the complement of the amoeba of the discriminant locus of a projection of an algebraic variety, there exist local Puiseux parametric equations of the variety. The series appearing in those parametric equations have support contained in cones which can be described in terms of the connected components of the complement of the amoeba. These cones are not necessarily strongly convex.

The results of Abhyankar \cite{Ab}, McDonald \cite{Mc}, Gonz\'alez P\'erez \cite{Go} and Aroca \cite{Ar} come as corollaries of the main result.

\section{Polyhedral convex cones}


 A set $\sigma\subseteq\mathbb{R}^{N}$ is said to be a
 \textbf{convex rational polyhedral cone} when it can be expressed in the form
$$\sigma =\{\lambda_{1}u^{(1)}+\lambda_{2}u^{(2)}+\cdots+\lambda_{M}u^{(M)} \ |\ \lambda_{j}\in\mathbb{R}_{\geq0}\},$$ where $u^{(1)},...,u^{(M)}\in\mathbb{Z}^{N}.$ The vectors $ u^{(1)},...,u^{(M)} $ are \textbf{a system of generators} of $ \sigma $ and we write,$$ \sigma=\langle u^{(1)},...,u^{(M)}\rangle. $$ 

 A cone is said to be \textbf{strongly convex} if it contains no linear subspaces of positive dimension.

As usual, here $ x\cdot y $ denotes the standard scalar product in $ \mathbb{R}^{N}. $
Let $\sigma\subset \mathbb{R}^{N}$ be a cone. Its \textbf{dual cone} $ \sigma^{\vee}$ is the cone given by 
  $$ \sigma^{\vee}:=\{x\in\mathbb{R}^{N}  \ |\  x\cdot u \geq 0,\ \forall u\in \sigma\}.$$
   Let $\mathrm{A}\subseteq \mathbb{R}^{n}$ be a non-empty convex set. The \textbf{recession cone} of A is the cone $$\mathrm{Rec}(A):= \{y\in \mathbb{R}^{n} \ |\  x+\lambda y \in \mathrm{A},\ \forall x\in A,\ \forall \lambda \geq 0\}.$$

  From now on, by a cone we will mean a convex rational polyhedral cone.
  \begin{rem}\label{rima} Given $ p\in \mathbb{R}^{N}$ and a cone $ \sigma\subset \mathbb{R}^{N}.$ One has,
$$ p+\sigma=\lbrace x\in \mathbb{R}^{N}\mid x\cdot\upsilon\geq p\cdot\upsilon,\: \forall \upsilon\in \sigma^{\vee}\rbrace. $$

\end{rem}

  
 A cone $ \sigma\subset \mathbb{R}^{N} $ is called a \textbf{regular cone} if it has a system of generators that is a subset of a basis of $ \mathbb{Z}^{N}.$  
  
 \begin{rem}\label{ryma} Any rational polyhedral cone $ \sigma\subset \mathbb{R}^{N} $ is a union of regular cones (see for example \cite[section 2.6]{fu}).
 \end{rem}
 For a matrix M, we denote by $ \mathrm{M}^{T} $ the transpose matrix of M.
 \begin{rem}\label{lastt} Take $ u^{(1)},...,u^{(N)}\in \mathbb{R}^{N}$ and let $ \mathrm{M }$ be the matrix that has as columns $ u^{(i)} $ $ i=1,...,N.$ Suppose that the determinant of $ \mathrm{M} $ is different from zero. If $ \sigma=\langle u^{(1)},...,u^{(N)}\rangle, $ then $ \sigma^{\vee}=\langle \upsilon^{(1)},...,\upsilon^{(N)}\rangle $ where $ \upsilon^{(i)} $ $ i=1,...,N $ are the columns of $ (\mathrm{M}^{-1})^{T}$  (see for example  \cite[Example 2.13.2.0.3]{Da}).
 \end{rem}

 \begin{lem}\label{last} Let $ \sigma=\langle u^{(1)},...,u^{(s)}\rangle\subset\mathbb{R}^{N}  $ be an s-dimensional regular cone and let $ (u^{(s+1)},...,u^{(N)}) $ be a  $\mathbb{Z}$-basis of $ \sigma^{\perp}\cap \mathbb{Z}^{N}.$ Let $ \mathrm{M} $ be the matrix that has $ u^{(1)},...,u^{(N)} $ as columns and let $ \upsilon^{(1)},...,\upsilon^{(N)}$ be the columns of $ (\mathrm{M}^{-1})^{T}.$ Then $$ \sigma^{\vee}=\langle \upsilon^{(1)},...,\upsilon^{(s)},\pm\upsilon^{(s+1)},...,\pm\upsilon^{(N)}\rangle.$$  
 \end{lem}
  
 \begin{proof} Clearly $ \langle \upsilon^{(1)},...,\upsilon^{(s)},\pm\upsilon^{(s+1)},...,\pm\upsilon^{(N)}\rangle\subseteq \sigma^{\vee}. $ Let $ \sigma^{\prime}:=\langle u^{(1)},...,u^{(N)}\rangle.$ Since $$\langle \upsilon^{(1)},...,\upsilon^{(N)}\rangle \subseteq \langle \upsilon^{(1)},...,\upsilon^{(s)},\pm\upsilon^{(s+1)},...,\pm\upsilon^{(N)}\rangle,$$ by Remark
 \ref{lastt},  $$ \langle \upsilon^{(1)},...,\upsilon^{(s)},\pm\upsilon^{(s+1)},...,\pm\upsilon^{(N)}\rangle^{\vee}\subseteq \sigma^{\prime}. $$ Now let $ x=\sum_{i=1}^{N}\lambda_{i}u^{(i)}$ be an element of $\langle \upsilon^{(1)},...,\upsilon^{(s)},\pm\upsilon^{(s+1)},...,\pm\upsilon^{(N)}\rangle^{\vee}.$ Since 
 
 $ \lambda_{i} u^{(i)}\cdot\upsilon^{(i)} \geq 0$ and  $ \lambda_{i} u^{(i)}\cdot -\upsilon^{(i)} \geq 0,$ we have that, $$ \lambda_{i}=0 \:\: \mathrm{for}\: i=s+1,...,N. $$ Then,
$ x=\sum_{i=1}^{s}\lambda_{i}u^{(i)}.$

 \end{proof}

\section{Amoebas}


Consider the map  \begin{equation}\label{eq}
\begin{array}{r@{\hspace{1 pt}} c@{\hspace{1 pt}}c@{\hspace{4pt}}l}
 \tau: &  \mathbb{C}^{N} &\longrightarrow &\mathbb{R}_{\geq 0}^{N} \\\\

& (z_{1},...,z_{N}) \ \ & \mapsto & (|z_{1}|,...,|z_{N}|).
\end{array}
\end{equation}

A set $\Omega\subseteq \mathbb{C}^{N}$ is called a \textbf{Reinhardt set} if $\tau^{-1}(\tau(\Omega))=\Omega.$\\
 Let log be the map defined by

 \begin{equation}\label{eqq}
\begin{array}{r@{\hspace{1 pt}} c@{\hspace{1 pt}}c@{\hspace{4pt}}l}
\mathrm{log}: & \mathbb{R}_{> 0}^{N}  &\longrightarrow & \mathbb{R}^{N} \\\\

& (x_{1},...,x_{N}) \ \ & \mapsto & (\mathrm{log}x_{1},...,\mathrm{log}x_{N}).
\end{array}
\end{equation}

We will denote $ \mu:=\mathrm{log}\circ \tau. $\\
 A Reinhardt set $\Omega\subseteq (\mathbb{C}^{\ast})^{N}$ is said to be \textbf{logarithmically convex} if the set
  $\mu(\Omega)$ is convex.\\ A Laurent polynomial is a finite sum of the form $ \sum_{(\alpha_{1},...,\alpha_{N})\in \mathbb{Z}^{N}}c_{\alpha}X^{\alpha} $ where $ c_{\alpha}\in\mathbb{C}.$  The generalization of Laurent polynomials are the Laurent series (for a further discussion of Laurent series, see for example, \cite{c4}). \\
   For a Laurent polynomial $f$ we denote by $ \mathcal{V}(f) $ its zero locus in $(\mathbb{C}^{\ast})^{N}.$
Given a Laurent polynomial $f,$ the \textbf{amoeba }of $ f $ is the image under $ \mu $ of the zero locus of $ f, $ that is,
 $$   \mathcal{A}_{f}:=\mu(\mathcal{V}(f)).$$ 

The notion of amoeba was introduced by Gelfand, Kapranov and Zelevinsky in \cite[Definition 1.4 ]{chi6}. The amoeba is a closed set with non-empty complement and each connected component $\mathcal{F}$ of the complement of the amoeba $\mathcal{A}_{f}$ is a convex subset \cite[Corollary 1.6]{chi6}.  From now on, by complement component we will mean connected component of the complement of the amoeba.
 For each complement component $\mathcal{F}$ of $\mathcal{A}_{f}$, we have that $\mu^{-1}(\mathcal{F})$ is a logarithmically convex Reinhardt domain.  An example of the amoeba of a polynomial is shown in Figure \ref{amiba}.

\begin{figure}[h]
\centering
\includegraphics[height=3.5cm , width=3.5cm]{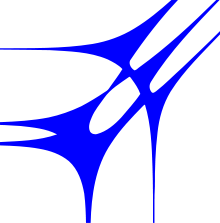}
\caption{\small{Amoeba of $f(x,y):=50x^{3}+83x^{2}y+24xy^{2}+y^{3}+392x^{2}+414xy+50y^{2}-28x+59y-100.$ (Taken from wikimedia commons, file: Amoeba4 400.png; Oleg Alexandrov). }}
\label{amiba}
\end{figure}


\begin{prop}\label{k} Let $\mathcal{F}$ be a complement component of $ \mathcal{A}_{f}.$ The fundamental group of $\mu^{-1}(\mathcal{F})$
is isomorphic to $\mathbb{Z}^{N}.$
\end{prop}

\begin{proof}
 Let $x:=(x_{1},...,x_{n})$ be a point of the complement component $ \mathcal{F}$. By
definition,
\begin{equation*}
 \begin{aligned}
\mu^{-1}(x)&=\{z=(z_{1},...,z_{N})\in\mathbb{C}^{N}:\mathrm{Log}(z)=x\}\\
&=\{z:(\mathrm{log}|z_{1}|,...,\mathrm{log}|z_{N}|)=(x_{1},...,x_{N})\}=\{z:|z_{i}|=e^{x_{i}}\}.
 \end{aligned}
 \end{equation*}
 That is,
$\mu^{-1}(x)$ is the product of $N$ circles of radius
$e^{x_{i}}$. The result follows from the fact that $ \mathcal{F} $ is contractible.
\end{proof}

\section{The Newton polytope and the order map}


Let $f=\sum_{\alpha\in\mathbb{Z}^{N}}c_{\alpha}z^{\alpha}$ be a Laurent series. The \textbf{set of exponents} of $f$ is the set $$\varepsilon(f):=\{\alpha\in\mathbb{Z}^{N} \ |\ c_{\alpha}\neq 0\}.$$ The set $ \varepsilon(f) $ is also called the \textbf{support} of $ f. $ When $ f$ is a Laurent polynomial, the convex hull of $\varepsilon(f)$ is called the \textbf{Newton polytope} of $f.$ We will denote the Newton polytope by $\mathrm{NP}(f).$ For $p\in \mathrm{NP}(f)$, the cone given by 
$$\sigma_{p}(\mathrm{NP}(f)):=\{\lambda(q-p):\lambda\in \mathbb{R}_{+}, q\in\mathrm{NP}(f) \}=\mathbb{R}_{+}(\mathrm{NP}(f)-p),$$ will be called the \textbf{cone associated to} $p.$ This cone is obtained by drawing half-lines from p through all points of $\mathcal{N}$ and then translating the result by $(-p)$ (see Figure \ref{Politopo}).

\begin{figure}[h]
\centering
\includegraphics[height=4.0cm , width=7.5cm]{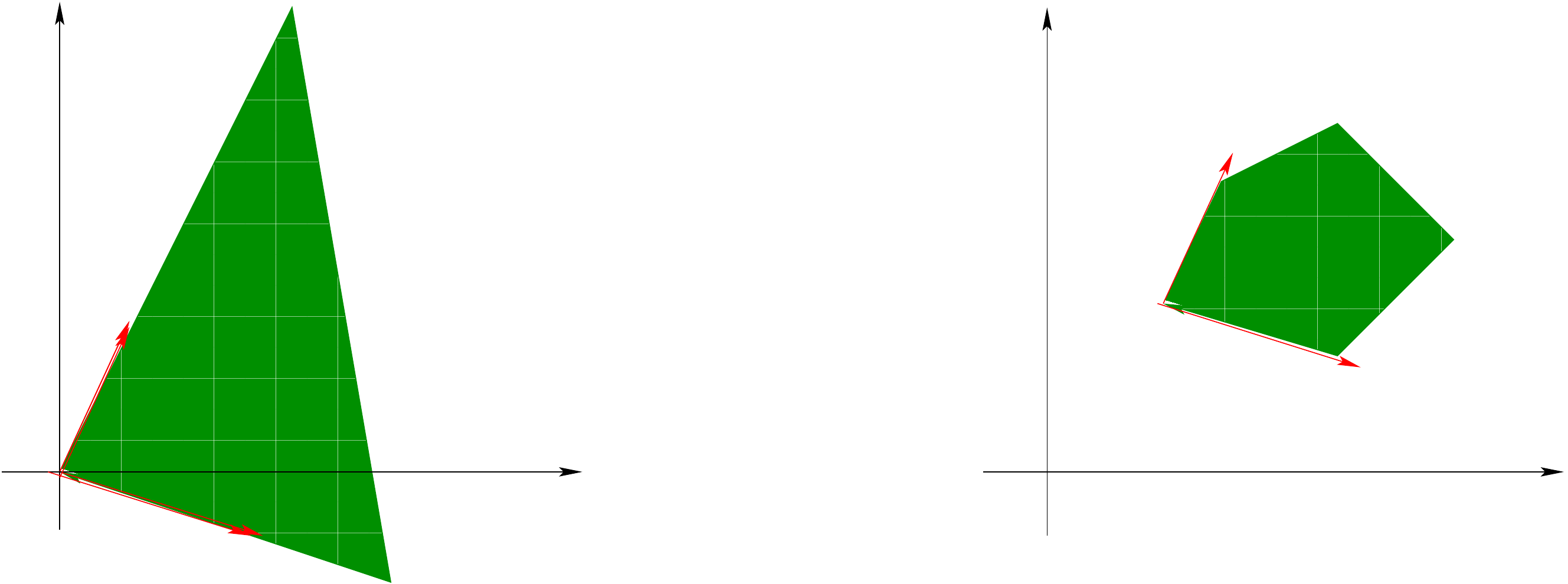}
\caption{\small{The cone associated to a vertex of the polygon on the right.}}
\label{Politopo}
\end{figure}

Forsberg, Passare and Tsikh gave in \cite{chi5} a natural correspondence between complement components of the amoeba $\mathcal{A}_{f}$ and integer points in $\mathrm{NP}(f)$ using the order map:

The \textbf{order map} is given by

 \begin{equation}\label{o}
\begin{array}{r@{\hspace{1 pt}} c@{\hspace{1 pt}}c@{\hspace{4pt}}l}
\mathrm{ord} : &  \mathbb{R}^{N}\setminus \mathcal{A}_{f} &\longrightarrow &\mathrm{NP}(f)\cap \mathbb{Z}^{N} \\\\

& x \ \ & \mapsto & \left( \frac{1}{(2\pi i)^{N}}\int_{\mu^{-1}(x)}\frac{z_{j}\partial_{j}f(z)}{f(z)}\frac{dz_{1}\cdots dz_{N}}{z_{1}\cdots z_{N}}\right)_{1\leq j\leq N}.
\end{array}
 \end{equation}

 Under the order map, points in the same complement component $\mathcal{F}$ of the amoeba $\mathcal{A}_{f}$ have the same value. This constant value is called the \textbf{order of} $\mathcal{F}$ and it is denoted by $\mathrm{ord}(\mathcal{F})$. The order map is illustrated in Figure \ref{kaya}.

\begin{prop}\label{order}
  The order map induces an injective map from the set of complement components of the amoeba $\mathcal{A}_{f}$ to $\mathrm{NP}(f)\cap \mathbb{Z}^{N}$. The vertices of $ \mathrm{NP}(f) $ are always in the image of this injection.
\end{prop}

\begin{proof} See \cite[Proposition 2.8]{chi5}.

\end{proof}

\begin{prop}\label{ver} Vertices of the Newton polytope are in bijection with those connected components of the complement of the amoeba which contain an affine convex cone (cone with vertex) with non-empty interior.

\end{prop}

\begin{proof} See \cite[Corollary 1.8]{chi6}.

\end{proof}

 Forsberg, Passare and Tsikh also gave in \cite{chi5} a relation between the order of a complement component of the amoeba and the recession cone of the component. They show that the recession cone of a complement component of order $ p$ is the opposite of the dual of the cone of the Newton polytope at the point $ p.$

 \begin{figure}[h]
\centering
\includegraphics[height=4.0cm , width=4.0cm]{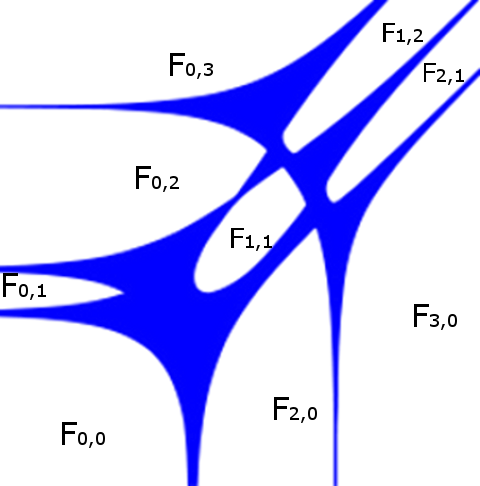}
\qquad  \qquad
\includegraphics[height=3.0cm , width=5.5cm]{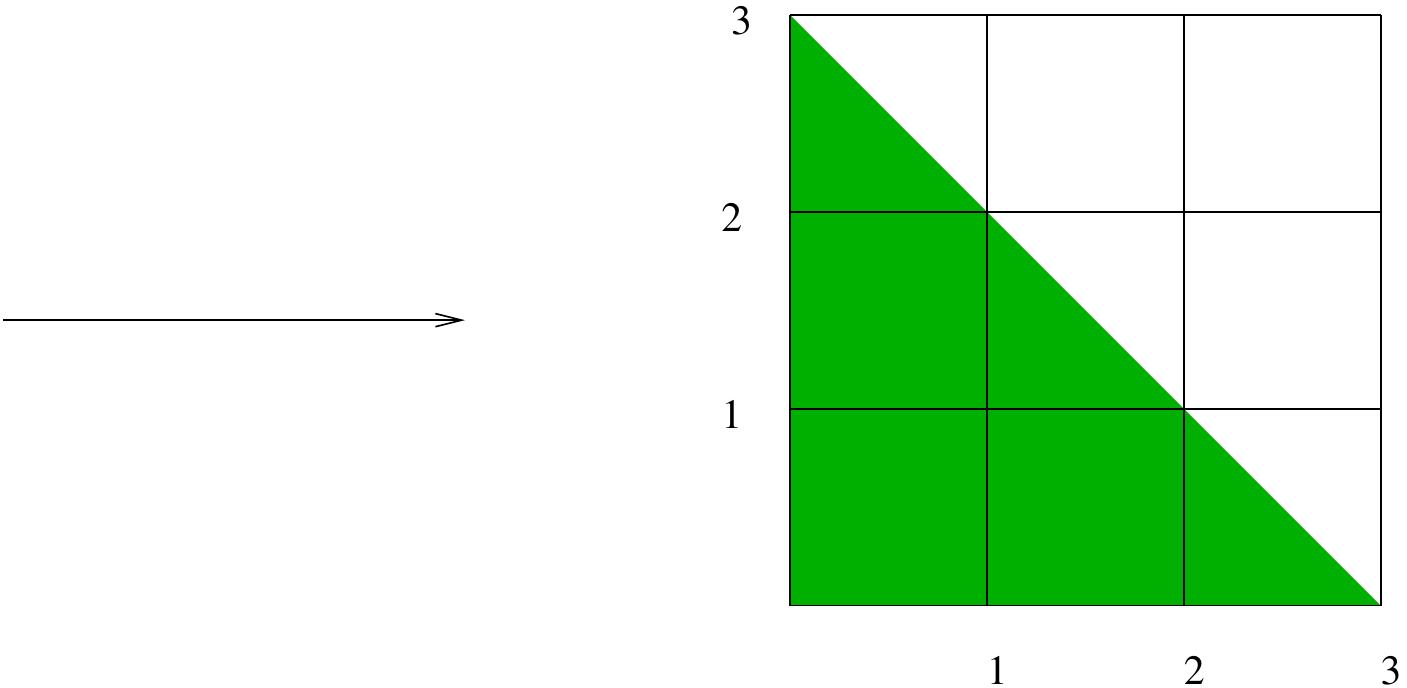}
\caption{\small{The order map between the complement components of $\mathcal{A}_{f}$ and $\mathrm{NP}(f)$ for $f$ as in Figure \ref{amiba}. (Here $\mathrm{F}_{i,j}$ denotes the complement component with order $(i,j)$)}.}
\label{kaya}
\end{figure}


\begin{prop}\label{re2} If $\mathcal{F}$ is a complement component of $\mathcal{A}_{f}$ then $$\sigma_{p}(\mathrm{NP}(f))= -\mathrm{Rec}(\mathcal{F})^{\vee},$$ where $ p=\mathrm{ord}(\mathcal{F}). $
\end{prop}
\begin{proof} The proposition is just a restatement of \cite[Proposition 2.6]{chi5}.
\end{proof}



\section{Toric morphisms}


Let $ \lbrace u^{(1)},...,u^{(N)}\rbrace\subset \mathbb{Z}^{N}$ be a $ N $-tuple of vectors which is a basis of
 $ \mathbb{Z}^{N}.$ Let M be the matrix that has $ u^{(1)},...,u^{(N)} $ as columns. Consider the map
 \begin{equation}\label{eq}
\begin{array}{r@{\hspace{1 pt}} c@{\hspace{1 pt}}c@{\hspace{4pt}}l}
  \Phi_{\mathrm{M}}: &  (\mathbb{C}^{\ast})^{N} &\longrightarrow & (\mathbb{C}^{\ast})^{N} \\\\

& z \ \ & \mapsto & (z^{u^{(1)}},z^{u^{(2)}},...,z^{u^{(N)}}).
\end{array}
\end{equation}

The map $\Phi_{\mathrm{M}}$ is an isomorphism with inverse $ \Phi_{\mathrm{M}^{-1}}.$

\begin{lem}\label{le1} Let $ \sigma\subset\mathbb{R}^{N} $ be a cone and $ p\in \mathbb{R}^{N}.$ If $\varrho\in (\mathbb{R}_{+}^{\ast})^{N}$ is such that $ \mu(\varrho)=p, $ then  $$ \mu^{-1}(p+\sigma)=\left\lbrace z\in(\mathbb{C}^{\ast})^{N}\mid |z|^{\upsilon}\geq \varrho^{\upsilon},\:\forall \upsilon\in\sigma^{\vee} \right\rbrace. $$

\end{lem}

\begin{proof}
   We have that
   
  \begin{equation*}
 \begin{aligned}
\mu^{-1}(p+\sigma) & \stackrel{\mathrm{Remark}\, \ref{rima}}{=} \left\lbrace z\in (\mathbb{C}^{\ast})^{N}\mid \mu(z)\cdot\upsilon\geq \mu(\varrho)\cdot\upsilon,\:\forall \upsilon\in\sigma^{\vee} \right\rbrace \\
& =  \left\lbrace  z\in (\mathbb{C}^{\ast})^{N} \: \big |\: e^{\sum_{i=1}^{N} \upsilon_{i}\mathrm{log}|z_{i}|} \geq e^{\sum_{i=1}^{N} \upsilon_{i}\mathrm{log}\varrho_{i}}     \right\rbrace\\
& =  \left\lbrace  z\in (\mathbb{C}^{\ast})^{N}\:\Big |\: \prod_{i=1}^{N}e^{\upsilon_{i}\mathrm{log}|z_{i}|}\geq \prod_{i=1}^{N}e^{\upsilon_{i}\mathrm{log}\varrho_{i}}        \right\rbrace\\
& =  \left\lbrace z\in (\mathbb{C}^{\ast})^{N}\:\Big |\: \prod_{i=1}^{N}|z_{i}|^{\upsilon_{i}}\geq \prod_{i=1}^{N}\varrho_{i}^{\upsilon_{i}},\:\forall \upsilon=(\upsilon_{1},...,\upsilon_{N})\in\sigma^{\vee} \right\rbrace.
\end{aligned} 
\end{equation*}

\end{proof}

\begin{prop}  Let $\mathrm{M }$ be as in (\ref{eq}) and let $ \sigma\subset \mathbb{R}^{N} $ be a cone. Given $ p\in \mathbb{R}^{N}$ one has, $$ \mu(\Phi_{\mathrm{M}}(\mu^{-1}(p+\sigma)))=q+\mathrm{M}^{T}\sigma,$$ where  $ \lbrace q\rbrace= \mu(\Phi_{\mathrm{M}}(\mu^{-1}(p))).$  
\end{prop}

\begin{proof}
We have that

\begin{equation*}
\begin{aligned}
\Phi_{\mathrm{M}}(\mu^{-1}(p+\sigma)) & \stackrel{\mathrm{Lemma}\,\ref{le1}}{=}\left\lbrace z\in (\mathbb{C}^{\ast})^{N}\mid |\Phi_{\mathrm{M}^{-1}}(z)|^{\upsilon}\geq \Phi_{\mathrm{M}^{-1}}(\rho)^{\upsilon},\:\forall \upsilon\in\sigma^{\vee}; \Phi_{\mathrm{M}}(\rho)=\varrho \right\rbrace\\
& = \left\lbrace  z\in (\mathbb{C}^{\ast})^{N}\mid  |z|^{\mathrm{M}^{-1}\upsilon}\geq \rho^{\mathrm{M}^{-1}\upsilon},\:\forall \upsilon\in\sigma^{\vee}\right\rbrace\\
& = \left\lbrace z\in (\mathbb{C}^{\ast})^{N}\mid  |z|^{w}\geq \rho^{w},\:\forall w\in \mathrm{M}^{-1}\sigma^{\vee}\right\rbrace
\end{aligned}
\end{equation*}

and
\begin{equation*}
\begin{aligned}
\mu( \Phi_{\mathrm{M}}(\mu^{-1}(p+\sigma))) & =\left\lbrace \mu(z)\in\mathbb{R}^{N}\mid w\cdot\mu(z)\geq \mathrm{log}(\rho^{w}),\:\forall w\in \mathrm{M}^{-1}\sigma^{\vee}\right\rbrace\\
& =\left\lbrace y\in\mathbb{R}^{N}\mid w\cdot y\geq \mathrm{log}(\rho^{w}),\:\forall w\in \mathrm{M}^{-1}\sigma^{\vee}\right\rbrace\\
& \stackrel{\mathrm{Remark}\, \ref{rima}}{=}q+(\mathrm{M}^{-1}\sigma^{\vee})^{\vee} = q+\mathrm{M}^{T}\sigma.
\end{aligned}
\end{equation*}

\end{proof}

\begin{cor} \label{cor}Let $ \Omega\subset(\mathbb{C}^{\ast})^{N} $ be a Reinhardt domain and let $ \sigma\subset\mathbb{R}^{N} $ be a cone. If $ \sigma\subset \mathrm{Rec}(\mu(\Omega))$  then $ \mathrm{M}^{T}\sigma\subset \mathrm{Rec}(\mu(\Phi_{\mathrm{M}}(\Omega))). $ 
\end{cor}

\begin{prop}\label{ryn} Let $ \sigma\subset\mathbb{R}^{N}$ be a cone. Let $ \varphi $ be a Laurent series and suppose that

 $ \varepsilon(\varphi)\subset p+\sigma$ where $ p\in\mathbb{R}^{N}.$ Then 
$ \varepsilon(\varphi\circ \Phi_{\mathrm{M}})\subset \mathrm{M}p+\mathrm{M}\sigma. $  
\end{prop}

\begin{proof}
It is enough to make the substitution.
\end{proof}

\section{Series development on Reinhardt domains}

It is well known that the Taylor development of a holomorphic function on a disc centered at the origin is a series with support in the non-negative orthant. In this section we will get a similar result for holomorphic functions on $ \xi_{d}^{-1}(\Omega) $ where $\Omega $ is a Reinhardt domain.

\begin{prop}\label{Laurent}  If $f(x)$ is a holomorphic function on a logarithmically convex Reinhardt domain, then there exists a (unique) Laurent series converging to $f(x)$ in this domain.
\end{prop}

\begin{proof} See for example \cite[Theorem 1.5.26]{c4}.
\end{proof}

\begin{lem}\label{pr} Let $ \Omega\subset(\mathbb{C}^{\ast})^{N} $ be a Reinhardt domain. Let $ (e^{(1)},...,e^{(N)}) $ be the canonical basis of $ \mathbb{R}^{N}$. If $ \langle -e^{(1)},...,-e^{(s)}\rangle\subset \mathrm{Rec}(\mu(\Omega)),$ then for every $ w\in \Omega,$ the $ s $-dimensional polyannulus
$$\mathbb{D}_{\tau(w),s}^{\ast}:=\left\lbrace z\in(\mathbb{C}^{\ast})^{N}\mid |z_{i}|\leq |w_{i}|; 1\leq i\leq s\: \mathrm{and}\: z_{i}= w_{i};s+1\leq i\leq N \right\rbrace $$  is contained in $ \Omega. $
 \end{lem}
 \begin{proof} Consider the cone $ \sigma:=\langle -e^{(1)},...,-e^{(s)}\rangle $ and take $ w\in \Omega. $ By Lemma \ref{last}, the dual cone of $ \sigma $ is
  $ \sigma^{\vee} =\langle -e^{(1)},...,-e^{(s)},e^{s+1},-e^{s+1},...,e^{N},-e^{N}\rangle.$ Since the cone
  
   $ \sigma\subset \mathrm{Rec}(\mu(\Omega)),$ we have that $ \mu(w)+\sigma \subset \mu(\Omega).$ Since $ \Omega $ is a Reinhardt domain, then
 \begin{equation*}
\begin{aligned}
\Omega\supseteq \mu^{-1}(\mu(w)+\sigma) & \stackrel{\mathrm{Lemma}\, \ref{le1}}{=}\left\lbrace z\in(\mathbb{C}^{\ast})^{N}\mid |z|^{\upsilon}\geq |w|^{\upsilon},\:\forall \upsilon\in\sigma^{\vee}\right\rbrace\\
& =\lbrace z\in (\mathbb{C}^{\ast})^{N}\:\vert\:   |z_{i}|\leq |w_{i}|,\: \mathrm{for}\: i=1,...,s;\\
& \mathrm{and}\; |z_{i}|=|w_{i}|,  \: \mathrm{for}\: i=s+1,...,N \rbrace\supset \mathbb{D}_{\tau(w),s}^{\ast}.
\end{aligned}
\end{equation*}

\end{proof}

Given $d,$ a natural number, set
 \begin{equation}\label{e}
\begin{array}{r@{\hspace{1 pt}} c@{\hspace{1 pt}}c@{\hspace{4pt}}l}
 \xi_{d}: &  \mathbb{C}^{N} &\longrightarrow &\mathbb{C}^{N} \\

& (z_{1},...,z_{N}) \ \ & \mapsto & (z_{1}^{d},...,z_{N}^{d})
\end{array}
 \end{equation}

\begin{lem}\label{cori} Let $ \Omega $ be a Reinhardt domain and suppose that $$(-\mathbb{R}_{\geq 0})^{s}\times \lbrace 0\rbrace^{N-s}\subset \mathrm{Rec}(\mu(\Omega)).$$ Let $f$ be a bounded holomorphic function on $ \xi_{d}^{-1}(\Omega).$ 
 Then the set of exponents of the Laurent series expansion of $f$ is contained in $ (\mathbb{R}_{\geq 0})^{s}\times \mathbb{R}^{N-s}.$
\end{lem}

\begin{proof}
Let $ \varphi=\sum_{I\in \mathbb{Z}^{N}}a_{I}z^{I}$ be the Laurent series expansion of $ f $ on $  \xi_{d}^{-1}(\Omega).$
Take $ w\in\Omega.$ By Lemma \ref{pr} we have that $ \mathbb{D}_{\tau(w),s}^{\ast}\subseteq \Omega.$ Therefore, $ \varphi $ is convergent and bounded on $\mathbb{D}_{\tau(\xi_{d}^{-1}(\tau(w))),s}^{\ast}.$ Let $ \pi $ be the map defined by,
 $$
\begin{array}{r@{\hspace{1 pt}} c@{\hspace{1 pt}}c@{\hspace{4pt}}l}
  \pi: &  (\mathbb{C}^{\ast})^{N} &\longrightarrow & (\mathbb{C}^{\ast})^{N-s} \\\\

& (z_{1},...,z_{N}) \ \ & \mapsto & (z_{s+1},z_{s+2},...,z_{N}).
\end{array}
$$
 The series in $ s $ variables,
$ \varphi_{w}:=\sum_{\alpha\in \mathbb{Z}^{s}}\psi_{\alpha}(\pi(w))z^{\alpha} $ where 
$$  \psi_{\alpha}(\pi(w))=\sum_{(\alpha,i_{s+1},...,i_{N})\in \varepsilon(\varphi)} a_{I}w_{s+1}^{i_{s+1}}\cdots w_{N}^{i_{N}}, $$ 
is convergent and bounded on the polyannulus 
$$\mathbb{D}^{\ast}:=\lbrace z\in (\mathbb{C}^{\ast})^{s}\mid |z_{i}|\leq \sqrt[d]{|w_{i}|}; 1\leq i\leq s\rbrace. $$ 
 By the Riemann removable singularity theorem (see for example \cite[Theorem\, 4.2.1]{c4}), there exists a (unique) holomorphic map that extends $ \varphi_{w} $ on $\mathbb{D}$ and $ \varphi_{w} $ is its Taylor development on that disc. Then $\varphi_{w}$ cannot have negative exponents.

\end{proof}

\begin{prop} \label{final} Let $ \Omega $ be a Reinhardt domain. Let $ \sigma\subset \mathbb{R}^{N} $ be a cone. Suppose that $
 \sigma\subseteq\mathrm{Rec}(\mu(\Omega)).$ Let $ f $ be a bounded holomorphic map on $ \xi_{d}^{-1}(\Omega).$ The set of exponents of the Laurent series expansion $ \varphi $ of $ f $ on $ \xi_{d}^{-1}(\Omega) $ is contained in $ -\sigma^{\vee}. $
\end{prop}

\begin{proof} By Remark \ref{ryma}, we can assume that $ \sigma $ is a regular cone. Let  $\lbrace \upsilon^{(1)},...,\upsilon^{(s)}\rbrace $ be the generator set of $\sigma. $ Let $ (\upsilon^{(s+1)},...,\upsilon^{(N)}) $ be a basis of $ \sigma^{\perp}\cap \mathbb{Z}^{N}.$ Let $ A $ be the matrix that has as columns $ \upsilon^{(i)} $ for $ i=1,...,N. $
Set $\mathrm{M}:=-(A^{-1})^{T}.$ Since $  \sigma\subset \mathrm{Rec}(\mu(\Omega)),$ by Corollary \ref{cor}, we have that 
 $$-A^{-1}\sigma=\langle -e^{(1)},...,-e^{(s)}\rangle\subseteq \mathrm{Rec}(\mu(\Phi_{\mathrm{M}}(\Omega))). $$
   The series $ h:=\varphi\circ \Phi_{\mathrm{M}}^{-1} $ is convergent in $ \Phi_{\mathrm{M}}(\xi_{d}^{-1}(\Omega)) $ then, by Lemma \ref{cori}, $$ \varepsilon(h)\subset \langle e^{(1)},...,e^{(s)},\pm e^{(s+1)},...,\pm e^{(N)}\rangle.$$ Therefore, by Proposition \ref{ryn} and Lemma \ref{last}, $$ \varepsilon(\varphi)\subseteq \mathrm{M}\langle e^{(1)},...,e^{(s)},\pm e^{(s+1)},...,\pm e^{(N)}\rangle=-\sigma^{\vee}.$$ 

\end{proof}

\section{Parameterizations compatible with a projection}

In what follows $\mathbb{X}\subseteq \mathbb{C}^{N+M}$ will denote an irreducible algebraic variety of dimension $ N $ such that the canonical projection
 \begin{equation}\label{ea}
\begin{array}{r@{\hspace{1 pt}} c@{\hspace{1 pt}}c@{\hspace{4pt}}l}
 \Pi: & \mathbb{X} &\longrightarrow &\mathbb{C}^{N} \\

& (z_{1},...,z_{N+M}) \ \ & \mapsto & (z_{1},...,z_{N}),
\end{array}
 \end{equation}
  is finite (that is, proper with finite fibers). In this case, there exists an algebraic set
$\mathcal{A}\subset\mathbb{C}^{N}$ such that, the restriction
$$\Pi:\mathbb{X}\setminus\Pi^{-1}(\mathcal{A})\longrightarrow\Pi(\mathbb{X})\backslash \mathcal{A} $$ is locally biholomorphic (see for example \cite[Proposition 3.7]{chi7} or \cite[\S 9]{chi15}). The intersection of all sets $\mathcal{A}$ with this property is called the \textbf{discriminant locus} of $\Pi.$ We denote the discriminant locus of $\Pi$ by $ \Delta. $


\begin{thm}\label{Teo} Let $ \mathbb{X},\,\Pi $ and $ \Delta $ be as above and let $ \mathcal{V}(\delta) $ be an algebraic  hypersurface of $ \mathbb{C}^{N} $ containing $ \Delta.$ Given a complement component $ \mathcal{F} $ of $ \mathcal{A}_{\delta},$ set  $\Omega:=\mu^{-1}(\mathcal{F}).$ For every connected component $ \mathcal{C}  $ of $\Pi^{-1}(\Omega)\cap \mathbb{X},$ there exists a natural number $d$ and a holomorphic morphism 
$ \Psi:\xi_{d}^{-1}(\Omega) \rightarrow \mathbb{C}^{M} $ such that
$$ \mathcal{C}= \lbrace (\xi_{d}(z),\Psi(z) )\mid z\in\xi_{d}^{-1}(\Omega) \rbrace. $$
\end{thm}




\begin{proof}  Note that $\Pi:\mathbb{X}\cap \Pi^{-1}(\Omega)\longrightarrow \Omega$ is locally biholomorphic. Let $\mathcal{C}$ be a connected component of
$\mathbb{X}\cap \Pi^{-1}(\Omega)$ and let $d$ be the cardinal of the generic fiber of  $\Pi|_{\mathcal{C}}$. Since both $\Pi|_{\mathcal{C}}$ and $\xi_{d}|_{\xi_{d}^{-1}(\Omega) }$ are locally biholomorphic, the pairs $(\mathcal{C},\Pi)$ and $(\xi_{d}^{-1}(\Omega) ,\xi_{d})$ are a $d$-sheeted and a $d^{N}$-sheeted covering of $\Omega$ respectively. Choose a point $z_{0}\in\Omega$, a point
$z_{1}\in\xi_{d}^{-1}(z_{0})$ and a point $z_{2}\in\Pi^{-1}(z_{0})\cap\, \mathcal{C}.$ Take the induced monomorphisms on the fundamental groups:
  $$ \xymatrix{
\pi_{1}(\xi_{d}^{-1}(\Omega),z_{1}) \ar[rd]_{\xi_{d_{\ast}}}& \pi_{1}(\mathcal{C},z_{2}) \ar[d]^{\Pi_{\ast}}\\
& \pi_{1}(\Omega,z_{0})
}$$

Note that: 
              \item i) An element $\gamma\in\pi_{1}(\Omega,z_{0})$ is in the subgroup
$\xi_{d_{\ast}}\pi_{1}(\xi_{d}^{-1}(\Omega) ,z_{1})$ if and only if
$\gamma=\alpha^{d}$ for some $\alpha\in\pi_{1}(\Omega,z_{0})$.\\

              \item ii) The index of $\Pi_{\ast}(\pi_{1}(\mathcal{C},z_{2}))$ in
$\pi_{1}(\Omega,z_{0})$ is equal to $d$ (see for example \cite[V$\S7$]{chi8}).\\

 By Proposition \ref{k}, $\pi_{1}(\Omega,z_{0})$ is abelian, then the cosets of
$\Pi_{\ast}(\pi_{1}(\mathcal{C},z_{2}))$ in $\pi_{1}(\Omega,z_{0})$
form a group of order $d$. By i) and ii), we have that for any
$\alpha$ in $\pi_{1}(\Omega,z_{0}),$ the element $\alpha^{d}$ belongs to
 $\Pi_{\ast}(\pi_{1}(\mathcal{C},z_{2}))$. Then,
$$\xi_{d_{\ast}}(\pi_{1}(\xi_{d}^{-1}(\Omega), z_{1}))\subseteq\Pi_{\ast}(\pi_{1}(\mathcal{C},z_{2})). $$ Applying the lifting lemma (see \cite[Theorem 5.1]{chi8}) we obtain a unique holomorphic morphism $\varphi$, such that $\varphi(z_{1})=z_{2}$ and the following diagram commutes

$$ \xymatrix{
 \xi_{d}^{-1}(\Omega) \ar[r]^{\varphi} \ar[rd]_{\xi_{d}} & \mathcal{C} \ar[d]^{\Pi}\\
   & \Omega
    }$$
The result follows from the fact that $ \varphi $ is a holomorphic morphism.
\end{proof}

\begin{rem}\label{re1}
  For every $P\in \Pi^{-1}(z_{0})\cap \mathcal{C}$ there exists a unique $\varphi$ as in the proof of Theorem \ref{Teo} such that $\varphi(z_{1})=P$. It follows that, there exist $d$ different morphisms $\varphi^{,}s $.
  
\end{rem}

\section{The series development of the parameterizations}




 Let $ \mathcal{C} $ be a connected component of $\Pi^{-1}(\Omega)\cap \mathbb{X}$ where $ \Pi,$ $\Omega$ and $ \mathbb{X} $ are as in Theorem \ref{Teo}. By the same theorem, there exists $ \varphi:\xi_{d}^{-1}(\Omega) \rightarrow \mathbb{C}^{N+M} $ of the form
  \begin{equation}\label{ia}
\begin{array}{r@{\hspace{1 pt}} c@{\hspace{1 pt}}c@{\hspace{4pt}}l}
 \varphi: & \xi_{d}^{-1}(\Omega)  &\longrightarrow & \mathbb{X} \\

& (z_{1},...,z_{N}) \ \ & \mapsto & (z_{1}^{d},...,z_{N}^{d},\varphi_{1},...,\varphi_{M}),
\end{array}
 \end{equation}
  where $\varphi_{i}:\xi_{d}^{-1}(\Omega) \longrightarrow \mathbb{C}$ is a
holomorphic function for $i=1,...,M$. Since $\xi_{d}^{-1}(\Omega) $ is a Reinhardt domain, by Proposition \ref{Laurent} we have, 

\begin{prop}\label{coro1} For every connected component of $\Pi^{-1}(\Omega)\cap \mathbb{X}$ there exist a natural number $d$ and  Laurent series $\phi_{i}$ converging to $ \varphi_{i} $ in $\xi_{d}^{-1}(\Omega) $ for $ i=1,...,M $, such that 

$(z_{1}^{d},...,z_{N}^{d},\phi_{1}(z),...,\phi_{M}(z))\in \mathbb{X},$  $\forall \,z\in \xi_{d}^{-1}(\Omega) $.
\end{prop}


  In fact, if $k$ is the degree of the projection $\Pi,$ by Remark \ref{re1}, there are $k$ M-tuples $(\phi_{1},...,\phi_{M})$ of convergent Laurent series such that 
  $$ (z_{1}^{d},...,z_{N}^{d},\phi_{1}(z),...,\phi_{M}(z))\in\mathbb{X},\, \forall \,z\in \xi_{d}^{-1}(\Omega) .$$

Now we describe the support set of the above Laurent series $ \phi_{i}. $



\begin{prop}\label{coness}
Let $\mathcal{F}$ be a complement component of $\mathcal{A}_{\delta}$ where $\delta$ is a polynomial
as in Theorem \ref{Teo}. Let $\phi_{i}$ for $ i=1,...,M $ be the Laurent series that converges to $ \varphi_{i} $ as in Proposition \ref{coro1}. Then $\varepsilon (\phi_{i})\subseteq  \sigma_{p}(\mathrm{NP}(\delta))$ for all $i=1,...,M,$ where $ p=\mathrm{ord}(\mathcal{F}).$
\end{prop}

\begin{proof} By Proposition \ref{re2}, we have that $ -\sigma_{p}(\mathrm{NP}(\delta))^{\vee}=\mathrm{Rec}(\mathcal{F}).$ Since every $ \varphi_{i} $ is a bounded holomorphic function on $ \xi_{d}^{-1}(\mu^{-1}(\mathcal{F})) $ the result follows from Proposition \ref{final}.
\end{proof}

\begin{thm}\label{teo1} Let $\mathbb{X}$ be an algebraic set in $\mathbb{C}^{N+M}$ with $\underline{0}\in\mathbb{X}$ and $\mathrm{dim}(\mathbb{X})=N.$ Let $\mathcal{V}(\delta)$ be an algebraic hypersurface containing the discriminant locus of the projection $\pi:\mathbb{X}\longrightarrow \mathbb{C}^{N}$. Then, given a complement component $\mathcal{F}$  of $\mathcal{A}_{\delta}$, there exist local parametric equations of $\mathbb{X}$ of the form
$$\qquad z_{i}=t_{i}^{d}\quad i=1,...,N,\qquad z_{N+j}=\phi_{j}(t_{1},...,t_{N})\quad j=1,...M,$$ where $d$ is a natural number, the $\phi_{j}$ are convergent Laurent series in $ \xi_{d}^{-1}(\mu^{-1}(\mathcal{F}))$ and their support is contained in the cone $-\mathrm{Rec}(\mathcal{F})^{\vee}.$
\end{thm}

\begin{proof} It is just a restatement of the Theorem \ref{Teo} and Proposition \ref{coness}.
\end{proof}

\begin{cor}\label{a}(Aroca, \cite{Ar}) Let $\mathbb{X}$ be an algebraic variety of  $\mathbb{C}^{N+M},$ $\underline{0}\in \mathbb{X}$, $\mathrm{dim}(\mathbb{X})=N.$ Let $U$ be a neighborhood of $\underline{0}$, and let $\pi$ be the restriction to $\mathbb{X}\cap U$ of the projection $(z_{1},...,z_{N+M})\mapsto (z_{1},...,z_{N}).$
Assume $\pi$ is a finite morphism. Let $\delta$ be a polynomial vanishing on the discriminant locus of $\pi.$
For each cone $\sigma$ of $NP(\delta)$ associated to a vertex, there exist $k\in \mathbb{N}$ and $M$ convergent Laurent series $s_{1},...,s_{M}$, such that $$\varepsilon(s_{i})\subseteq\sigma, \quad i=1,..,M, \quad and \quad f(z_{1}^{k},...,z_{N}^{k},s_{1}(z),...,s_{M}(z))=0$$ for any $f$ vanishing on $\mathbb{X},$ and any $z$ in the domain of convergence of the $s_{i}.$
\end{cor}

\begin{proof} By Proposition \ref{order}, for every vertex $V $ of $NP(\delta)$ there exists a complement component of $\mathcal{A}_{\delta}$ with order $V.$ Then, by Theorem \ref{teo1} there exist convergent Laurent series with support in the cone associated to the complement component of order $V.$

\end{proof}

\begin{cor}\label{mc}(McDonald, \cite{Mc}) Let $ F(x_{1},...,x_{N},y)=0 $ be an algebraic equation with complex coefficients. There exists a fractional power series expansion (Puiseux series) $ \phi(x_{1},...,x_{N}) $ such that $$ F(x_{1},...,x_{N},\phi)=0 $$ and the support of $ \phi $ is contained in some strongly convex polyhedral cone.

\end{cor}
\begin{proof} The result follows by applying a similar argument as in the proof of Corollary \ref{a}.

\end{proof}

\begin{rem}
P.D. Gonz\'alez P\'erez showed in \cite{Go}, with an additional hypothesis, that the supporting cone in Corollary \ref{mc} can be chosen to be a cone of the Newton polytope of the discriminant of the polynomial defining the hypersurface with respect to $ y $. Thus, in this sense, by the proof of Corollary \ref{mc} we also get this result.
\end{rem}


\begin{cor}(Abhyankar-Jung \cite{Ab}) Let $\underline{0}$ be a quasi-ordinary singularity of a complex algebraic set
 $\mathbb{X}\subseteq\mathbb{C}^{N+M},$ $\mathrm{dim}(\mathbb{X})=N.$ Then, there exists a natural number $d$ and $M$ convergent power series $\phi_{1},...,\phi_{M}$ such that
$$\qquad z_{i}=t_{i}^{d},\quad i=1,...,N,\quad z_{N+j}=\phi_{j}(t_{1},...,t_{N}),\quad j=1,...,M,$$ are parametric equations of $\mathbb{X}$ about $\underline{0}.$
\end{cor}

\begin{proof} By definition of quasi-ordinary singularity, the discriminant locus of the projection $\pi$ is contained in the coordinate hyperplanes, then it is contained in the algebraic hypersurface defined by $\beta (z):=z_{1}\cdots z_{N}.$ Note that $NP(\beta)$ has just one cone contained in the non-negative orthant. The result follows from Theorem \ref{teo1}.
\end{proof}

\begin{ej}
Consider the hypersurface defined by $ f:=z^{2}-x-y+1 $ in $ \mathbb{C}^{3}.$ The discriminant of the polynomial $ f $ is $ \Delta:=x+y-1. $
By the generalized binomial theorem, we have a series expansion of $ \Delta^{1/2} $ 

\begin{equation*}
\begin{aligned}
\varphi_{1}:=\sum_{k=0}^{\infty}\left(_{k}^{1/2}\right)  y^{k}(-1+x)^{1/2-k}
=\sum_{k=0}^{\infty}\left(_{k}^{1/2}\right)\sum_{j=0}^{\infty}\left(_{j}^{1/2-k}\right)(-1)^{j}y^{k}x^{1/2-k-j}. 
 \end{aligned}
 \end{equation*}
 We know by Theorem \ref{teo1} that $ \varphi_{1} $ converges in $ \mu^{-1}(\mathrm{F})$ for some complement component F of the amoeba.
Since $ \varphi_{1} $ converges in the region $ |x|> 1, $ $ |y|< |x|-1$ and this region is map under $ \mu $ to the complement component \textbf{A} of the amoeba (see figure \ref{ame}), we have that $ \varphi_{1} $ converges in $  \mu^{-1}(\textbf{A}).$ Since $ \sigma_{(1,0)} $  is the unique cone of the NP($ \Delta) $ such that a translation of $ -\sigma_{(1,0)}^{\vee} $  is contained in the complement component \textbf{A}, by Proposition \ref{ver} this complement component is associated to the vector $ (1,0). $ According to the Theorem \ref{teo1} we must have that $$ \varepsilon(\varphi_{1})\subset \langle
(-1,0),(-1,1)\rangle $$ which is true, because

$$ \varepsilon(\varphi_{1})=\left\lbrace (1/2-k-j,k)\mid j,k\in \mathbb{N}\cup \lbrace 0\rbrace \right\rbrace  $$ 
and $$(1/2-k-j,k)=j-1/2(-1,0)+k(-1,1).$$  Reasoning analogously as before, we get another series expansion of $ \Delta^{1/2},$

\begin{equation*}
\begin{aligned}
\varphi_{2}:=\sum_{k=0}^{\infty}\left(_{k}^{1/2}\right)  x^{k}(-1+y)^{1/2-k}
=\sum_{k=0}^{\infty}\left(_{k}^{1/2}\right)\sum_{j=0}^{\infty}\left(_{j}^{1/2-k}\right)(-1)^{j}x^{k}y^{1/2-k-j}, 
 \end{aligned}
 \end{equation*}
which converges in the region $ |y|> 1,$  $|x/y-1|< 1. $ Therefore, $\varphi_{2}$ converges in $ \mu^{-1}(\textbf{B}).$ As before, by Proposition \ref{ver} we can see that this complement component is associated to the vector $ (0,1).$ Therefore, by Theorem \ref{teo1} we must have that 
$$ \varepsilon(\varphi_{2})\subset \langle (0,-1),(1,-1)\rangle.$$ This is true because 
$$ \varepsilon(\varphi_{2})=\left\lbrace (k,1/2-k-j)\mid j,k\in \mathbb{N}\cup \lbrace 0\rbrace \right\rbrace $$ and
$$ (k,1/2-k-j)=j-1/2(0,-1)+k(1,-1). $$  Analogously for

 $$ \varphi_{3}:=\sum_{k=0}^{\infty}\left(_{k}^{1/2}\right)(-1)^{1/2-k}(x+y)^{k}. $$ We have that $ \varphi_{3} $ converges in $\mu^{-1}(\textbf{C}).$ The complement component \textbf{C} is associated to the vector $ (0,0) $ and the support of $ \varphi_{3} $ is contained in the non-negative orthant.
 

\end{ej}

\begin{figure}[h]
\centering
\includegraphics[height=5.0cm , width=5.0cm]{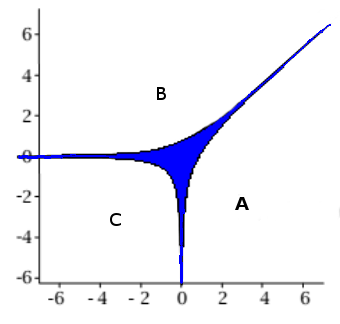}
\caption{\small{The amoeba of $x+y-1$ (Taken from wikimedia commons; Oleg Alexandrov).}}
\label{ame}
\end{figure}

Fuensanta Aroca\\
Universidad Nacional Aut\'onoma de M\'exico\\
Unidad Cuernavaca\\
A.P.273-3 C.P. 62251\\
Cuernavaca, Morelos\\
M\'exico\\
fuen@matcuer.unam.mx \\

V\'ictor Manuel Saavedra\\
Universidad Nacional Aut\'onoma de M\'exico\\
Unidad Cuernavaca\\
A.P.273-3 C.P. 62251\\
Cuernavaca, Morelos\\
M\'exico\\
victorm@matcuer.unam.mx

 \end{document}